%% file: article.tex
\documentclass[a4paper,11pt]{amsart}

\input{general_commands}

\input{special_commands}

\newtheorem{theorem}{Theorem}
\newtheorem{lemma}[theorem]{Lemma}
\newtheorem{proposition}[theorem]{Proposition}
\newtheorem{corollary}[theorem]{Corollary}
\newtheorem{remark}[theorem]{Remark}

\usepackage{dutchcal}

\numberwithin{equation}{section}
\numberwithin{theorem}{section}

\usepackage{a4wide}

\usepackage[foot]{amsaddr}

\usepackage[breaklinks,bookmarks=false]{hyperref}
\hypersetup{%
colorlinks,%
linkcolor=blue,%
citecolor=blue,%
urlcolor=blue,%
plainpages=false,%
pdfwindowui=false,%
pdfstartview={FitH}%
}

\title{A ill-posed scattering problem saturating Weyl's law}
\author{T. Chaumont-Frelet$^\dagger$}

\address{\vspace{-.5cm}}
\address{\noindent \tiny \textup{$^\dagger$Inria, Univ. Lille, CNRS, UMR 8524 -- Laboratoire Paul Painlev\'e}}

\begin{document}

\maketitle

\begin{abstract}
This paper focuses on the well-posedness (or lack thereof) of three-dimensional time-harmonic
wave propagation problems modeled by the Helmholtz equation. It is well-known that if the problem
is set in bounded domain with Dirichlet boundary conditions, then the Helmholtz problem is
well-posed for all (real-valued) frequencies except for a sequence of countably many resonant
frequencies that accumulate at infinity. In fact, if the domain is sufficiently smooth, this
can be quantified further and Weyl's law states that the number of resonant frequencies less than
a given $\omega > 0$ scales as $\omega^3$. On the other hand, scattering problems set in $\R^3$
with a radiation condition at infinity and a bounded obstacle modeled by variations
in the PDE coefficients are well-posed for all frequencies under mild regularity
assumption on such coefficients. In 2001, Filinov provided a counter example of
a rough coefficient such that the scattering problem is not well-posed for (at least)
a single frequency $\omega$. In this contribution, we use this result to show that
for all $\varepsilon > 0$ we can design a rough coefficient corresponding to a compactly
supported obstacle such that the scattering problem is ill-posed for a countable sequence of
frequencies accumulating at infinity, and such that the number of such frequencies less
than any given $\omega > 0$ scales as $\omega^{3-\varepsilon}$.
\end{abstract}

\section{Introduction}

Time-harmonic wave propagation problems play a crucial role in many areas of
physics and engineering. At the mathematical level, the simplest corresponding
PDE model is the Helmholtz equation. There, the unknown $u$ is a complex-amplitude
satisfying
\begin{equation}
\label{eq_helmholtz_intro}
-\omega^2 \mu u - \div(\AAA\grad u) = f
\end{equation}
where $\mu$ and $\AAA$ are known coefficients encoding the properties of the
propagation medium and $f$ is a given right-hand side that represents the
source of excitation that generates the wave. $\mu$ is a scalar valued functions
that is uniformly bounded away from $0$ and $+\infty$ and $\AAA$ is a symmetric
tensor-valued function that is uniformly elliptic and bounded. It is then of
central interest to understand whether this modeling leads to a well-posed
problem when the PDE in~\eqref{eq_helmholtz_intro} is supplemented with boundary conditions.

For interior problems, the PDE in~\eqref{eq_helmholtz_intro} needs to be
satisfied weakly in a bounded domain $\Omega \subset \R^3$, and we additionally
impose the Dirichlet condition that $u = 0$ on the boundary $\partial \Omega$
(Neumann or mixed boundary conditions could also be considered without significant
complications). In this case, it is fairly clear that the Helmholtz problem
if and only if $\omega^2$ does not correspond to an eigenvalue of Dirichlet
Laplacian in $\Omega$, i.e., a real number $\lambda$ such that
\begin{equation}
\label{eq_eigen_intro}
-\div(\AAA\grad \psi) = \lambda \mu \psi
\end{equation}
for some function $\psi \neq 0$. In most applications, $\mu$ and $\AAA$
are measurable, and uniformly bounded and elliptic. Then, standard spectral
theory gives a fully satisfactory framework to analyse such eigenvalues
(see e.g. \cite[Chapters 6 and 9.8]{brezis_2010a}).
In addition, if $\mu$, $\AAA$, and the boundary $\partial \Omega$ are
sufficiently smooth, then Weyl's law gives a precise characterization
of the distribution of eigenvalues~\cite{agmon_1968a,weyl_1911a}. Specifically, this law states
that the number of eigenvalues less than any given $\lambda \geq 0$
scales as $\lambda^{3/2}$ as $\lambda \to +\infty$. Correspondingly,
the number of frequencies less than $\omega$ for which the Helmholtz
problem is ill-posed scales as $\omega^3$.

On the other, in scattering problems, we assume that $\mu = 1$
and $\AAA = \III$ outside some compact set $D \subset \R^3$,
and we demand that the PDE in~\eqref{eq_helmholtz_intro} is
satisfies a.e. in $\R^3$. In this case, we close the problem
with the Sommerfeld radiation condition
\begin{equation}
\label{eq_sommerfeld_intro}
\partial_r u - i\omega u = o(r^{-1}) \text{ as } r \to +\infty
\end{equation}
at infinity, where $r$ is the radial variable. In contrast with
interior problems for which resonant frequencies are physically
relevant, one naturally expect that the scattering
problem~\eqref{eq_helmholtz_intro}-\eqref{eq_sommerfeld_intro} is well-posed
for all (real) frequencies $\omega$. In fact, under fairly mild
assumptions on the coefficients $\mu$ and $\AAA$, combining Rellich's
theorem (see, e.g. \cite[Lemmas 9.8 and 9.9]{mclean_2000a}) and the unique
continuation principle (UCP, see \cite{burq_1998a,carleman_1933a}),
we can show that this is indeed the case. For instance, if $\mu$ and $\AAA$ are piecewise
Lipschitz onto a finite partition~\cite{ball_capdebosq_tseringxiao_2012a}, then it is true that
for all $\omega \geq 0$, there exists a unique $u$ satisfying~\eqref{eq_helmholtz_intro}
and~\eqref{eq_sommerfeld_intro}.

A natural question is then whether the aforementioned smoothness assumptions are necessary
to ensure the well-posedness of~\eqref{eq_helmholtz_intro}-\eqref{eq_sommerfeld_intro}.
More specifically, the question lies in whether the UCP holds true even for rough coefficients.
The answer is shown to be no. Specifically, the author of~\cite{filonov_2001a} constructs
coefficients $\hmu$ and $\hAAA$ for which there exits an eigenpair $(\hlambda,\widehat{\psi})$
solving~\eqref{eq_eigen_intro} with a compactly supported eigenfunction $\widehat{\psi}$.
A direct consequence is that the corresponding scattering problem cannot be well-posed
for the frequency $\widehat{\omega} = \hlambda^{1/2}$. All in all~\cite{filonov_2001a}
provides an example of scattering problem that fails to be well-posed for (at least)
one frequency.

In this work, we build on top of the result in~\cite{filonov_2001a}, and construct
coefficients $\mu$ and $\AAA$ for which the scattering problem fails to be well-posed
for countably many frequencies. In fact, for any $\varepsilon > 0$, we can construct
coefficients $\mu_\varepsilon$ and $\AAA_\varepsilon$ such that the number of frequencies
less than $\omega$ for which the problem is not well-posed scales as $\omega^{3-\varepsilon}$.
In other word, the number of (real) resonant frequencies for this scattering problem
essentially scales as the one of an interior problem. Naturally, these coefficients are
uniformly elliptic and bounded, and coincide with $1$ and $\III$ outside a compact set.

Our construction simply relies on scaling and translation, and can be easily described
as follows. Imagine rescaling the coefficients $\hmu$ and $\hAAA$ by a factor two 
and naming the result $\hAAA_2$ and $\hmu_2$. Then, then the accordingly rescaled eigenfunction
$\widehat{\psi}_2$ is an eigenfunction for the coefficients $\hmu_2$ and $\hAAA_2$
with eigenvalue $\hlambda_2 = 2\hlambda$. Because $\widehat{\psi}$ and $\widehat{\psi}_2$
both have compact support, we can now translate the coefficients $\hmu_2$ and $\hAAA_2$
and the eigenfunction $\widehat{\psi}_2$ in such a way that its support does not overlap
with the one of $\widehat{\psi}$. Defining now coefficients $\AAA$ and $\mu$ that
are equal to $\hmu$ and $\hAAA$ on the support of $\widehat{\psi}$ and to (the translated
version of) $\hmu_2$ and $\hAAA_2$ on the (translated) support of $\widehat{\psi}_2$,
we obtain coefficients for which both $\hlambda$ and $4\lambda$ are eigenvalues with
compactly supported eigenfunctions. The crux of our analysis then consists in packing
as many cubes as possible into a bounded domain, and to translate the coefficients
as above. To be more precise, we show that we can arrange a countable sequence of
cubes $\{Q_{\ck}\}_{\ck \geq 1}$ with sides $\ell_\ck = \ck^{-(1+\varepsilon)/3}$
into a compact set of $\R^3$, which provides the desired result. This is intuitive
since the total volume of such cubes is finite because the series $\ck^{-1-\varepsilon}$
converges. Yet, precisely arranging the cubes and establishing the result is technical.

The remainder of this work is organized as follows. Section~\ref{section_notation}
collects notation and preliminary results. In Section~\ref{section_result} we establish
our main results under the assumption that suitable packing of cube is provided. Finally,
we establish the required cube packing result in Section~\ref{section_cube}.

\section{Notation and preliminary results}
\label{section_notation}

\subsection{Scalar, vector and tensors}

We employ the standard italic font ($a,b,\dots$) to denote scalar quantities.
Vectors are denoted with a bold font ($\ba,\bb,\dots$) and tensors with an
underline ($\AAA,\tens{B},\dots$). The symbol $|\cdot|$ may be used for scalars,
vectors and tensors. For scalars, it is the modulus, for vectors the euclidean
norm, and for tensors, the matrix norm induced by the euclidean norm.

\subsection{Function spaces}

Hereafter, for an open set $D \subset \R^3$, $L^2(D)$
and $H^1(\Omega)$ are the usual (complex-valued) Lebesgue
and Sobolev spaces, see~\cite{brezis_2010a}.
We denote by $L^2_{\rm c}(D)$ the subset of $L^2(D)$
collecting functions with compact support. We will also employ
the standard $L^\infty(D)$ for the (real-valued) Lebesgue spaces
of essentially bounded functions, and denote by $\LLL^\infty(D)$
the set of tensor-valued functions (i.e. mapping $D$ to $\R^{3 \times 3}$
that belong element wise to $L^\infty(D)$. We finally employ the
(real-valued) fractional Sobolev spaces $W^{s,p}(D)$ for $0 < s < 1$
and $p \geq 1$, and $\WWW^{s,p}(D)$ contains tensor-valued functions
whose entries belong to $\WWW^{s,p}(D)$. These fractional Sobolev
spaces are analyzed, e.g., in~\cite[Chapter 3]{mclean_2000a}.

We classically denote by $C^0(\R^3)$ the space of real-valued continuous function
and, for $0 < \alpha \leq 1$, $C^{0,\alpha}(\R^3)$ is the space of H\"older continuous
functions. We refer the reader to~\cite[Section 5.1]{evans_1998a} for a description of these spaces.
We denote by $\CCC^0(\R^3)$ and $\CCC^{0,\alpha}(\R^3)$ the set of tensor-valued functions
which are element-wise $C^0(\R^3)$ and $C^{0,\alpha}(\R^3)$.

\subsection{Cubes and affine maps}

For $\bx \in \R^3$ and $\ell > 0$, the notation
\begin{equation*}
Q(\bx,\ell) \eq (\bx_1,\bx_1+\ell) \times (\bx_2,\bx_2+\ell) \times (\bx_3,\bx_3+\ell)
\end{equation*}
denotes the (open) cube of side $\ell$ with corner $\bx$. We also
use the short-hand notation $\hQ \eq Q(\bzero,1)$ for the reference cube.
The affine map $F: \R^3 \to \R^3$ defined by
\begin{equation}
\label{eq_affine_map}
F(\hbx) \eq \bx+\frac{1}{\ell}\hbx \qquad \forall \hbx \in \R^3
\end{equation}
is invertible. It is bijective between $\hQ$ and $Q(\bx,\ell)$,
and its Jacobian is given by
\begin{equation}
\label{eq_jacobian}
(\JJJ F)(\hbx) = \frac{1}{\ell} \III.
\end{equation}

\subsection{Reference coefficient}

As alluded to earlier, our work rely on the following
key result is established in~\cite{filonov_2001a}.

\begin{proposition}[Reference coefficient]
\label{proposition_reference_coefficient}
There exists a uniformly bounded and elliptic, continuous symmetric tensor-valued coefficient
$\hAAA \in \LLL^\infty(\R^3)$ with $\supp(\hAAA-\III) \subset \hQ$ and a function $u \in H^1(\R^3)$
with $\supp \hu \subset \hQ$ such that
\begin{equation}
\label{eq_reference_coefficient}
-\div(\hAAA \grad \hu) = \hlambda \hu \text{ in } \R^3
\end{equation}
for some $\hlambda > 0$. In addition, we have $\hAAA \in \CCC^{0,\alpha}(\R^3)$
for all $0 < \alpha < 1$.
\end{proposition}

In fact,~\cite[Theorem 1]{filonov_2001a} gives the existence of
some $\tu \in H^1(\R^3)$ with compact support $S$ and some $\tAAA$ that does
not necessarily correspond to $\III$ outside a compact set so that~\eqref{eq_reference_coefficient}
holds true for some $\tlambda > 0$. However, we can always modify $\tAAA$ outside the support $S$
of $\tu$ so that it coincides with $\III$ outside a compact set $\widetilde{S} \subset S$,
and rescale to obtain the desired result in $\hQ$, i.e., in such a way that $\widetilde{S}$
is rescaled inside $\hQ$.

\subsection{Cube packing}

Our main results are obtained by combining a technical cube packing result
with Proposition~\ref{proposition_reference_coefficient}. We state this result
here and provide its proof in Section~\ref{section_cube} below.

\begin{lemma}[Cube packing]
\label{lemma_cube_packing}
For all $\varepsilon > 0$, there exists $a^\varepsilon \in \R_+$
and a sequence of corners $\{\bx_{\ck}^\varepsilon\}_{\ck \geq 1} \subset \R^3$
such that, letting $Q_{\ck}^\varepsilon \eq Q(\bx_{\ck}^\varepsilon,\ck^{-(1+\varepsilon)/3})$
for $\ck \geq 1$, we have
\begin{equation*}
Q_{\cn}^\varepsilon \cap Q_{\cm}^\varepsilon = \emptyset 
\end{equation*}
for all $\cn,\cm \in \N$ with $\cn \neq \cm$, and
\begin{equation*}
Q_{\ck}^\varepsilon \subset Q^\varepsilon \eq (0,1) \times (0,1) \times (0,a^\varepsilon)
\end{equation*}
for all $\ck \geq 1$.
\end{lemma}

\begin{remark}[Dyadic partition]
It is possible to obtain infinitely many resonant frequencies with
simpler arrangements of cubes. For instance, one could employ a dyadic
partition of the unit interval with cubes $\{Q_{\ck}\}_{\ck \geq 1}$ of
size $2^{-\ck}$. However, in this case, the distribution of the eigenvalues
is much sparser than what is predicted by Weyl's law, since the number of
resonant frequencies less than a given $\lambda \geq 0$ would scale as
$\log_4\lambda$.
\end{remark}

\section{Main result}
\label{section_result}

As already advertised, our coefficient is built by rescaling the reference
coefficient $\hAAA$ of Proposition~\ref{proposition_reference_coefficient}
into the cubes provided by Lemma~\ref{lemma_cube_packing}. Its precise
construction and key properties are given in Lemma~\ref{lemma_coefficient}.

\begin{lemma}[Coefficient]
\label{lemma_coefficient}
For $\ck \geq 1$, let denote by $F_{\ck}^\varepsilon$ the affine mapping between
$\hQ$ and $Q_{\ck}^{\varepsilon}$ as per~\eqref{eq_affine_map}. We then define
$\AAA^\varepsilon: \R^3 \to \R^{3 \times 3}$ as follows. Let $\bx \in \R^3$.
If there exists $\ck \geq 1$ such that $\bx \in Q_{\ck}^\varepsilon$, then we let
\begin{equation}
\label{eq_definition_coefficient}
\AAA^\varepsilon(\bx) \eq \hAAA((F_{\ck}^\varepsilon)^{-1}(\bx))
\end{equation}
and let $\AAA^\varepsilon(\bx) \eq 1$ otherwise. Then,
$\AAA^\varepsilon$ is uniformly bounded and elliptic. Besides,
$\AAA^\varepsilon \in \LLL^\infty(\R^3)$ but
$\AAA^\varepsilon \notin \CCC^0(\R^3)$. In fact,
if $0 < s < 1$ and $p \geq 2$, $\AAA^\varepsilon \notin \WWW^{s,p}(Q^\varepsilon)$
whenever $sp \geq 3\varepsilon/(1+\varepsilon)$.
\end{lemma}

\begin{proof}
Since the cubes $\{Q_{\ck}^\varepsilon\}_{\ck \geq 1}$ from Lemma~\ref{lemma_cube_packing}
are non-overlapping, the definition in~\eqref{eq_definition_coefficient} indeed properly defines
$\AAA^\varepsilon$. It is also clear that $\AAA^\varepsilon$ is uniformly bounded and
elliptic since $\hAAA$ is. To show that $\AAA^\varepsilon \in \LLL^\infty(\R^3)$ it remains
to show that it is measurable, and we proceed as follows.
Consider a measurable set $V \subset \R^{3 \times 3}$. If
$\III \in V$, we let $U_0 \eq \R^3 \setminus \bigcup_{\ck} Q_{\ck}^\varepsilon$
and $U_0 \eq \emptyset$ otherwise. Then, we have
\begin{equation*}
(\AAA^\varepsilon)^{-1}(V)
=
U_0 \cup \bigcup_{\ck \geq 1} (\hAAA \circ (F_{\ck}^\varepsilon)^{-1})^{-1}(V).
\end{equation*}
We first observe that $U_0$ is measurable. Besides, each $\hAAA \circ (F_{\ck}^\varepsilon)^{-1}$
is measurable, meaning that each set in the countable union is measurable. This shows
that $(\AAA^\varepsilon)^{-1}(V)$ is measurable, which completes the proof.

To show that $\AAA^\varepsilon$ is not continuous, we first show that it is not uniformly
continuous. To do so, we consider two reference points $\hbx,\hby \in \R^3$ such that
$\hAAA(\hbx) \neq \hAAA(\hby)$. Such points necessarily exist, since
Proposition~\ref{proposition_reference_coefficient} could not hold if $\hAAA$ was constant
due to the unique continuation principle, see e.g.~\cite{burq_1998a}.

For $\ck \geq 1$, let us now consider the physical points
$\bx_\ck \eq T_{\ck}^\varepsilon(\bx)$ and $\by_\ck \eq T_{\ck}^\varepsilon(\by)$.
Clearly, we have
\begin{equation*}
|\bx_{\ck}-\by_{\ck}| \leq \ck^{-(1+\varepsilon/3)} |\hbx-\hby|
\end{equation*}
while
\begin{equation*}
|\AAA^\varepsilon(\bx_{\ck})-\AAA^\varepsilon(\by_{\ck})| = |\hAAA(\hbx)-\hAAA(\hby)| > 0.
\end{equation*}
As a result,
\begin{equation*}
|\AAA^\varepsilon(\bx_{\ck})-\AAA^\varepsilon(\by_{\ck})|
\geq
\frac{|\hAAA(\hbx)-\hAAA(\hby)|}{|\hbx-\hby|}
\ck^{(1+\varepsilon)/3} |\bx_{\ck}-\by_{\ck}|.
\end{equation*}
This shows that $\AAA^\varepsilon$ is not uniformly continuous.
Because $\AAA^\varepsilon$ is constant outside a bounded set,
it follows that $\AAA^\varepsilon$ is not continuous, see e.g.~\cite[Theorem 4.19]{rudin_1976a}.

The proof that $\AAA^\varepsilon \notin \WWW^{s,p}(Q^\varepsilon)$ under the above
assumption is similar. Since $\hAAA \in \CCC^{0,s}(\R^3)$, we have $\hAAA \in \WWW^{s,p}(\hQ)$.
We then note that that because $\hAAA$ is not constant, its semi-norm
\begin{equation*}
|\hAAA|_{\WWW^{s,p}(\hQ)}^p
\eq
\int_{\hby \in \hQ}\int_{\hbx \in \hQ}
\frac{|\hAAA(\hbx)-\hAAA(\hby)|^p}{|\hbx-\hby|^{sp+3}}
d\hbx d\hby
=
\widehat{c} > 0
\end{equation*}
does not vanish. A simple scaling argument then tells us that for all $\ck \geq 1$, we have
\begin{equation*}
|\AAA^\varepsilon|_{\WWW^{s,p}(Q_{\ck}^\varepsilon)}^p
=
\widehat{c}
\left (\ck^{-(1+\varepsilon)/3}\right )^{3-sp}
=
\widehat{c}\ck^\theta
\end{equation*}
with
\begin{equation*}
\theta \eq -\frac{1-\varepsilon}{3}(3-sp)
\end{equation*}
The corresponding series converges when we sum over $\ck \geq 1$ iif $\theta < -1$,
which is equivalent to the condition
\begin{equation*}
sp < \frac{3\varepsilon}{1+\varepsilon}.
\end{equation*}
\end{proof}

We are now ready to establish our main result, whereby we construct an
infinite number of eigenpairs, with all eigenfunctions support in the
same compact set. The associating counting function essentially saturates
Weyl's law.

\begin{theorem}[Compactly supported eigenfunctions]
Fix $\varepsilon > 0$. Then, for all $\ck \geq 1$,
there exists $u_{\ck} \in H^1(\R^3)$ with
$\supp u_{\ck}^\varepsilon \subset Q_{\ck}^\varepsilon \subset Q^\varepsilon$
such that
\begin{subequations}
\label{eq_eigen_modes}
\begin{equation}
-\div(\AAA^\varepsilon \grad u_{\ck}^\varepsilon) = \lambda_{\ck}^\varepsilon u_{\ck}
\end{equation}
with
\begin{equation}
\lambda_{\ck}^\varepsilon \eq \hlambda \ck^{2(1+\varepsilon)/3}.
\end{equation}
\end{subequations}
In particular, we have
\begin{equation}
\label{eq_counting_function}
\sharp
\left \{
\lambda_{\ck}^\varepsilon \leq \lambda \; | \; \ck \geq 1
\right \}
=
\left \lfloor
\left (\frac{\lambda}{\hlambda}\right )^{\frac{2}{3(1+\varepsilon)}}
\right \rfloor.
\end{equation}
\end{theorem}

\begin{proof}
The proof of~\eqref{eq_eigen_modes} relies on a simple scaling argument. For each $\ck \geq 1$,
we let $u_{\ck} \eq u \circ (F_{\ck}^\varepsilon)^{-1}$ where $F_{\ck}^\varepsilon$
is the affine map from Lemma~\ref{lemma_coefficient}. It is then clear that
$u_{\ck} \in H^1(\R^3)$ with $\supp u_{\ck} \subset Q_{\ck}^\varepsilon \subset Q^\varepsilon$.
A simple consequence of~\eqref{eq_jacobian} is that
\begin{equation*}
\JJJ((F_{\ck}^\varepsilon)^{-1}) = \ck^{-(1+\varepsilon)/3}\III.
\end{equation*}
Then,~\eqref{eq_eigen_modes} follows by the chain rule since
\begin{equation*}
-\div(\AAA^\varepsilon \grad u_{\ck})
=
-\ck^{-(1+\varepsilon)2/3} \left \{\div(\hAAA \grad \hu)\right \} \circ (F_{\ck}^\varepsilon)^{-1}
=
\ck^{(1+\varepsilon)3/2} \hlambda u_{\ck}.
\end{equation*}

The proof of~\eqref{eq_counting_function} is also fairly straightforward.
Let $\lambda > 0$. If $\ck_\star \geq 1$ is such that $\lambda_{\ck_\star} \leq \lambda$,
then there are (at least) $\ck_\star$ eigenvalues smaller than $\lambda$.
The estimate in~\eqref{eq_counting_function} by picking the largest $\ck_\star$
such that $\lambda_{\ck_\star} \leq \lambda$.
\end{proof}

As a direct consequence, we obtain a familiy of ill-posed scattering problems.

\begin{corollary}[Ill-posed scattering problem]
There exists coefficients $\mu \in L^\infty(\R^d)$ and $\AAA \in \LLL^\infty(\R^d)$
\begin{equation*}
\mu(\bx) = 1 \qquad \AAA(\bx) = \III
\end{equation*}
whenever $|\bx| \geq R$ with the property that the scattering problem of finding
$u \in H^1_{\rm loc}(\R^3)$ such that
\begin{equation*}
\left \{
\begin{array}{rcll}
-\omega^2 \mu u-\div(\AAA\grad u) &=& f & \text{ in } \R^3
\\
\partial_r u -i\omega u &=& o(|\bx|^{-1}) & \text{ as } |\bx| \to +\infty
\end{array}
\right .
\end{equation*}
for a given $f \in L^2_{\rm c}(\R^3)$ is not well-posed whenever
there exists $\ck \geq 1$ such that
\begin{equation*}
\omega\diam = \widehat{c} \ck^{(1+\varepsilon)/3},
\end{equation*}
where $\widehat{c}$ is a generic constant.
\end{corollary}

\section{Cube packing}
\label{section_cube}


This section establishes Lemma~\ref{lemma_cube_packing}. For all $\varepsilon > 0$,
we show that we can pack a countable set of cubes
$\{ Q(\bx_{\ck}^\varepsilon,\ck^{-(1+\varepsilon)/3}) \}_{\ck \geq 1}$ with properly
selected corners $\{\bx_{\ck}^\varepsilon\}_{\ck \geq 1} \subset \R^3$ in a bounded
set without overlaps. This result seems intuitively true since the total
volumes of such a set of cubes is finite. However, this consideration on
the volume does not guarantee on its own that we can pack the cubes in
a bounded domain. Here, we rigorously show that this is indeed the case.

\subsection{Construction of the corners}

Roughly speaking, our construction simply amounts to packing the cubes
along lines, then arranging such lines together in squares, and finally
stacking these layers on top of each other.

It will be convenient to set $\ell_{\ck} \eq \ck^{-(1+\varepsilon)/3}$ for $\ck \geq 1$.
We let $\bx_1^\varepsilon = \bzero$, and, for $\ck \geq 1$,
\begin{equation}
\label{eq_cube_construction}
\arraycolsep=0.5pt
\bx_{\ck+1}
\eq
\left |
\begin{array}{rrrrrrll}
(&\bx_{\ck,1}^\varepsilon + \ell_{\ck},&\bx_{\ck,2}^\varepsilon&,&\bx_{\ck,3}^\varepsilon&)
&
\text{ if }
\bx_{\ck,1}^\varepsilon + \ell_{\ck} + \ell_{\ck+1} \leq 1,\text{ else}
\\
(&0,&\bx_{\ck,2}^\varepsilon& + \ell_{\ck},&\bx_{\ck,3}^\varepsilon&)
&
\text{ if }
\bx_{\ck,2}^\varepsilon + \ell_{\ck} + \ell_{\ck+1} \leq 1,\text{ and}
\\
(&0,&&0,&\bx_{\ck,3}^\varepsilon&+\ell_{\ck})
&
\text{ otherwise.}
\end{array}
\right .
\end{equation}

It is easily seen that the $\{\bx_{\ck}^\varepsilon\}_{\ck \geq 1}$ generated
by~\eqref{eq_cube_construction} are sufficiently far apart that
the cubes do not overlaps. By construction, it is clear that
$0 \leq \bx_{\ck,j}^\varepsilon-\ell_\ck/2$ and $\bx_{\ck,j}^\varepsilon + \ell_\ck \leq 1$
for $j=1$ and $2$. We also obviously have $\bx_{\ck,3}^\varepsilon \geq 0$,
so that the only thing that remains to be checked is that there exists
$a_\varepsilon \in \R$ such that $\bx_{\ck,3}^\varepsilon+\ell_\ck \leq a_\varepsilon$
for all $\ck \geq 1$.

\subsection{Plan of the proof}

We are in fact going to show that the construction introduced in~\eqref{eq_cube_construction}
provides a satisfactory answers for $\varepsilon < 1/2$. If $\varepsilon \geq 1/2$, we can simply
use the corners obtained by~\eqref{eq_cube_construction} for any value of
$\widetilde \varepsilon < 1/2$, say $\widetilde \varepsilon \eq 1/4$. Indeed, we then have
$Q(\bx_{\ck}^\varepsilon,\ck^{-(1+\varepsilon)/3}) \subset Q(\bx_{\ck}^\varepsilon,\ck^{-(1+\widetilde \varepsilon)/3})$
since $\ck^{-(1+\varepsilon)/3} \leq \ck^{-(1+\widetilde \varepsilon)/3}$, and the packing obtained
for the bigger cubes obviously works for the smaller cubes.
In the remainder, we thus fix $\varepsilon < 1/2$.

We let $\CS_1 \eq 1$, and we denote by $1 < \CS_2 < \CS_3 < \dots$ the indices
for which
\begin{equation*}
\cy_{\CS_{\cl}} = (0,0,\cy_{\CS_{\cl}-1,3}+\ell_{\CS_{\cl}-1}), \quad \cl > 1.
\end{equation*}
For $\cl \geq 1$, $\CS_{\cl}$ is thus the index of the cube that starts the $\cl^{\rm th}$
vertical layer in the construction. Note that there is indeed an infinite number of
layers, since the area of the faces of the cube is infinite. Indeed,
\begin{equation*}
\sum_{\ck \geq 1} \ell_{\ck}^2
=
\sum_{\ck \geq 1} \ck^{-(2+2\varepsilon)/3}
\geq
\sum_{\ck \geq 1} \ck^{-1}
=
+\infty
\end{equation*}
as $\varepsilon \leq 1/2$. The height of the $\cl^{\rm th}$ layer is then given by
\begin{equation}
\label{eq_cube_definition_CH}
\CH_\cl \eq (\CS_\cl)^{-(1+\varepsilon)/3}
\end{equation}
and Lemma~\ref{lemma_cube_packing} holds true if
\begin{equation*}
\sum_{\cl \geq 1} \CH_\cl < +\infty.
\end{equation*}

Our goal is therefore to give lower bounds for $\CS_\cl$, for $\cl \geq \cl_\star$
sufficiently large. To that end, we will characterize explicitly $\{\CS_\cl\}_{\cl \geq 1}$
as a recurrent sequence, and we proceed as follows.

First, for $\cn \leq \cm$, we let
\begin{equation}
\label{eq_cube_definition_CL}
\CL_{\cn}^{\cm} \eq \sum_{\ck=\cn}^{\cm} \ck^{-(1+\varepsilon)/3}.
\end{equation}
denote the length required to arrange cubes of size $\{\ell_\ck\}_{\ck=\cn}^{\cm}$
in a straight line by gluing them face to face. For $\cn \in \N$, the number $\CM_{\cn}$
of cubes of size $\{\ell_{\ck}\}_{\ck=\cn}^{\CM_{\cn}}$ that can be arranged in a line of
length $1$ is then given by
\begin{equation}
\label{eq_cube_definition_CM}
\CM_{\cn} \eq \max \{ \cm \geq \cn \; | \; \LL_{\cn}^{\cm} \leq 1 \} - \cn + 1.
\end{equation}

Then, for a fixed cube $Q_{\cn}^\varepsilon$, $\cn \in \N$, we denote by
\begin{equation}
\label{eq_cube_definition_CB}
\CB^{\cn}_1 \eq \cn, \qquad \CB^{\cn}_{r+1} = \CB^{\cn}_r + \CM_{\CB^{\cn}_r}
\end{equation}
the index of the cube at start of $r^{\rm th}$ row if the square is initiated
with the cube $Q_{\cn^\varepsilon}$. As a result,
\begin{equation}
\label{eq_cube_definition_CW}
\CW_{\cn}^{R}
\eq
\sum_{r=1}^R (\CB^{\cn}_r)^{-(1+\varepsilon)/3}
\end{equation}
is the width occupied by $R$ rows starting at the cube $Q_{\cn}^\varepsilon$, and
we further denote by
\begin{equation}
\label{eq_cube_definition_CR}
\CR_{\cn} \eq \max \{R \geq 1 \; | \; \CW_{\cn}^R \leq 1 \}
\end{equation}
is the number of rows than fit in the unit square starting from $Q_{\cn}^\varepsilon$.
The number of cubes that constitute those rows is then
\begin{equation}
\label{eq_cube_definition_CN}
\CN_{\cn} \eq \CB^{\cn}_{\CR_{\cn}+1}-\cn+1.
\end{equation}
Equipped with this definition, we can write that
\begin{equation}
\label{eq_cube_definition_CS}
\CS_{\cl+1} = \CS_\cl + \CN_{\CS_{\cl}}
\end{equation}
for $\cl \geq 1$.

\subsection{Preliminary results}

For completeness, we provide here some elementary result on recurrent sequences.

\begin{proposition}[Sum bound]
\label{proposition_cube_sum_bound}
Consider two integers $\cn,\cm \in \Z$ with $\cn \leq \cm$,
and a non-increasing function $f: C^0([\cn-1,\cm],\R)$.
Then, we have
\begin{equation}
\label{eq_cube_sum_bound}
\sum_{\ck = \cn}^\cm f(\ck) \leq \int_{\ck-1}^{\cm} f(\tau) d\tau.
\end{equation}
\end{proposition}

\begin{proof}
The proof relies on the simple observation that the sum can be rewritten has
\begin{equation*}
\sum_{\ck = \cn}^\cm f(\ck) = \int_{\ck-1}^{\cm} f(\lceil \tau \rceil) d\tau.
\end{equation*}
Then, the conclusion follows from the fact that $f(\lceil \tau \rceil) \leq f(\tau)$
due to the monotonicity assumption on $f$.
\end{proof}

\begin{proposition}[Recurrent sequence]
\label{proposition_cube_bound_sequence}
Fix $\ck_\star \in \N$, $\rho > 0$ as well as $0 < \alpha < 1$,
and consider the sequence $\{\cx_\ck\}_{\ck \geq \ck_\star}$ such that
\begin{equation*}
\cx_{\ck+1} \geq \cx_{\ck} + \rho \cx_{\ck}^\alpha
\end{equation*}
for all $\ck \geq \ck_\star$. For all $0 < \eta < 1$, there exists
$\mu_\star(\alpha,\rho,\eta) > 0$ such that if
$x_{\ck_\star} \geq \mu_\star(\alpha,\rho,\eta)$, then
\begin{equation}
\label{eq_cube_bound_sequence}
\cx_\ck
\geq
\left (
(1-\alpha)(1-\eta) \rho (\ck-\ck_\star)+\cx_{\ck_\star}^{1-\alpha}
\right )^{1/(1-\alpha)}
\end{equation}
for all $\ck \geq \ck_\star$.
\end{proposition}

\begin{proof}
Define the sequence
\begin{equation*}
\cy_\ck
\eq
\left (
(1-\alpha)(1-\eta) \rho (\ck-\ck_\star)+\cx_{\ck_\star}^{1-\alpha}
\right )^{1/(1-\alpha)}
\end{equation*}
for $\ck \geq \ck_\star$. We first observe that $\cy_{\ck_\star} = \cx_{\ck_\star}$.
Besides, it is clear that the sequence $\{\cy_{\ck}\}_{\ck \geq \ck_\star}$ increases
so that $\cy_{\ck} \geq \cx_{\ck_\star}$ for all $\ck \geq \ck_\star$.
Elementary computations reveal that
\begin{equation*}
\cy_{\ck+1}
=
\cy_\ck
\left (
1+ (1-\alpha)(1-\eta) \rho \cy_\ck^{\alpha-1}
\right )^{1/(1-\alpha)}.
\end{equation*}

At that point, if we can choose $\mu_\star(\alpha,\rho,\eta)$ in such a way that
\begin{equation*}
\cy_{\ck+1}
\leq
\cy_\ck
\left (
1+ \rho \cy_\ck^{\alpha-1}
\right )
=
\cy_{\ck} + \rho \cy_\ck^\alpha,
\end{equation*}
then, the conclusion follows by induction. It therefore remains to properly select
$\mu_\star(\alpha,\rho,\eta)$.

We start by considering $\nu_\star(\eta) \in \R$ such that
\begin{equation*}
(1+\tau)^{1/(1-\alpha)} \leq 1 + \frac{1}{1-\eta} \frac{1}{1-\alpha} \tau
\end{equation*}
for all $\tau \leq \nu_\star(\eta)$, and we define
\begin{equation*}
\mu_\star(\alpha,\rho,\eta)
\eq 
\left (\frac{\nu_\star(\eta)}{(1-\alpha)(1-\eta)}\right )^{1/(\alpha-1)}.
\end{equation*}
At that point, we simply need to check that $\cx_{\ck_\star} \geq \mu_\star(\alpha,\rho,\eta)$
guarantees that
\begin{equation*}
(1-\alpha)(1-\eta) \rho \cy_\ck^{\alpha-1}
\leq
\nu_\star(\eta).
\end{equation*}
However, since the sequence $\cy_\ck \geq \cx_{\ck_\star}$ for $\ck \geq \ck_\star$
and $\alpha < 1$, we have
\begin{equation*}
(1-\alpha)(1-\eta) \rho \cy_\ck^{\alpha-1}
\leq
(1-\alpha)(1-\eta) \rho \cx_{\ck_\star}^{\alpha-1}
\leq
(1-\alpha)(1-\eta) \rho \mu_\star(\alpha,\rho,\eta)^{\alpha-1}
=
\nu_\star(\eta).
\end{equation*}
\end{proof}

The following simple inequalities will be useful
\begin{equation}
\label{eq_cube_lower_bound_1ta}
(1+\tau)^\alpha \geq 1+\alpha\tau
\end{equation}
for all $\alpha \geq 1$ and $\tau \geq 0$. The bound
\begin{equation}
\label{eq_cube_upper_bound_1ta}
(1+\tau)^\alpha \leq 1+\alpha\tau
\end{equation}
also holds true for $0 \leq \alpha \leq 1$ and $\tau \geq 0$.

\subsection{Proof of Lemma~\ref{lemma_cube_packing}}

With these preliminary results established, we can proceed with the main proof.

\begin{lemma}[Lower bound on $\CM_{\cn}$]
\label{lemma_cube_CM}
There exists $\cn_\star(\varepsilon) \in \N$ such that
\begin{equation}
\label{eq_cube_bound_CM}
\CM_{\cn} \geq (1-\varepsilon) \cn^{(1+\varepsilon)/3}
\end{equation}
for all $\cn \geq \cn_\star(\varepsilon)$.
\end{lemma}

\begin{proof}
We start by invoking~\eqref{eq_cube_sum_bound}, giving
\begin{equation*}
\CL_{\cn}^{\cm}
\leq
L(\cm)
\eq
\frac{3}{2-\varepsilon} (\cm^{(2-\varepsilon)/3}-(\cn-1)^{(2-\varepsilon)/3}).
\end{equation*}

The function $\R_+ \ni \tau \to L(\tau) \in \R$ is increasing.
Besides, it is also easy to check that $L(\tau_\star) = 1$ for
\begin{equation*}
\tau_\star \eq \left (\frac{2-\varepsilon}{3}+(\cn-1)^{(2-\varepsilon)/3}\right )^{3/(2-\varepsilon)}.
\end{equation*}
It follows that setting $\cm_\star \eq \lfloor \tau_\star \rfloor \geq x_\star-1$,
we have $\CL_{\cn}^{\cm_\star} \leq 1$. By definition of $\CM_{\cn}$, we have
\begin{equation*}
\CM_{\cn}
\geq
\cm_\star
\geq
\left (\frac{2-\varepsilon}{3}+(\cn-1)^{(2-\varepsilon)/3}\right )^{3/(2-\varepsilon)} - \cn.
\end{equation*}
Since $3/(2-\varepsilon) > 1$, we can apply~\eqref{eq_cube_lower_bound_1ta}, and we have
\begin{align*}
\left (\frac{2-\varepsilon}{3}+(\cn-1)^{(2-\varepsilon)/3}\right )^{3/(2-\varepsilon)}
&=
(\cn-1)
\left (1+\frac{2-\varepsilon}{3}(\cn-1)^{-(2-\varepsilon)/3}\right )^{3/(2-\varepsilon)}
\\
&\geq
(\cn-1)\left (1+(\cn-1)^{-(2-\varepsilon)/3}\right ),
\end{align*}
so that
\begin{equation*}
\CM_{\cn}
\geq
(\cn-1)^{(1+\varepsilon)/3}-1
=
\left (
\left ( 1-\frac{1}{\cn} \right )^{1+\varepsilon/3}
-
\cn^{-(1+\varepsilon)/3}
\right )
\cn^{(1+\varepsilon)/3}
\end{equation*}
and the conclusion follows by selecting $\cn_\star(\varepsilon) \in \N$ such that
\begin{equation*}
\left ( 1-\frac{1}{\cn_\star(\varepsilon)} \right )^{1+\varepsilon/3}
-
\cn_\star(\varepsilon)^{-(1+\varepsilon)/3}
\geq
1-\varepsilon.
\end{equation*}
\end{proof}

\begin{lemma}[Lower bound on $\CB_r^{\cn}$]
\label{lemma_cube_CB}
There exists $\cn_\star(\varepsilon) \in \N$ such that
\begin{equation}
\label{eq_cube_bound_CB}
\CB_r^{\cn} \geq \left (
\frac{2}{3}(1-\varepsilon)^2(r-1)+\cn^{(2-\varepsilon)/3}
\right )^{3/(2-\varepsilon)}
\end{equation}
for all $\cn \geq \cn_\star(\varepsilon)$ and $r \geq 1$.
\end{lemma}

\begin{proof}
Let us consider $\cn$ large enough that Lemma~\ref{lemma_cube_CM} applies.
Since $\CM_{\cn} \geq 1$, the sequence $\{\CB_r^{\cn}\}_{r \geq 1}$ is increasing.
Therefore, $\CB_r^{\cn} \geq \cn$ and we can apply~\eqref{eq_cube_bound_CM} for all $r \geq 1$.
This leads to
\begin{equation*}
\CB^{\cn}_{r+1}
\geq
\CB^{\cn}_r + (1-\varepsilon) \left (\CB^{\cn}_r\right )^{(1+\varepsilon)/3}
\end{equation*}
for all $r \geq 1$. Since $\CB_1^{\cn} = \cn$, we can then chose $\cn_\star(\varepsilon) \in \N$
so that Proposition~\ref{proposition_cube_bound_sequence} applies with $\ck_\star = 1$,
$\rho = 1-\varepsilon$, $\alpha = (1+\varepsilon)/3$ and $\eta = \varepsilon$.
Then,~\eqref{eq_cube_bound_sequence} provides~\eqref{eq_cube_bound_CB}, since then
$1-\alpha = (2-\varepsilon)/3$.
\end{proof}

\begin{lemma}[Upper bound on $\CW_{\cn}^R$]
There exists $\cn_\star(\varepsilon) \in \N$ such that
\begin{equation}
\label{eq_cube_bound_CW}
\CW_{\cn}^R \leq (R+\varepsilon)\cn^{-(1+\varepsilon)/3}
\quad
\forall R \geq 1
\end{equation}
for all $\cn \geq \cn_\star(\varepsilon)$.
\end{lemma}

\begin{proof}
Consider $R \in \N$. Recalling the definition of $\CW_{\cn}^R$ in~\eqref{eq_cube_definition_CW}
and the lower bound on~$\CB_r^{\cn}$ established in~\eqref{eq_cube_bound_CB}, we have
\begin{equation*}
\CW_{\cn}^R
\leq
\sum_{r=1}^R
\left (
\frac{2}{3}(1-\varepsilon)^2(r-1)+\cn^{(2-\varepsilon)/3}
\right )^{-(1+\varepsilon)/(2-\varepsilon)}.
\end{equation*}
Since the summand is a decreasing function of $r$, Proposition~\ref{proposition_cube_sum_bound},
and~\eqref{eq_cube_sum_bound} says that
\begin{equation}
\label{tmp_cube_int_bound}
\CW_{\cn}^R \leq F(R)-F(0),
\end{equation}
where, for $\tau \geq 0$,
\begin{equation*}
F(\tau)
\eq
\frac{3}{2} \frac{1}{(1-\varepsilon)^2} \frac{2-\varepsilon}{1-2\varepsilon}
\left (
\frac{2}{3}(1-\varepsilon)^2(\tau-1)+\cn^{(2-\varepsilon)/3}
\right )^{(1-2\varepsilon)/(2-\varepsilon)}
\end{equation*}
is the anti-derivative of the summand. We can further rewrite $F$ as
\begin{equation}
\label{tmp_cube_F_G}
F(\tau)
=
\frac{3}{2} \frac{1}{(1-\varepsilon)^2} \frac{2-\varepsilon}{1-2\varepsilon}
\cn^{(1-2\varepsilon)/3}
G(\tau),
\end{equation}
with
\begin{equation*}
G(\tau) \eq \left (
1+\frac{2}{3}(1-\varepsilon)^2(\tau-1)\cn^{-(2-\varepsilon)/3}
\right )^{(1-2\varepsilon)/(2-\varepsilon)}
\end{equation*}
for all $\tau \geq 0$.

On the one hand, we have the bound
\begin{equation}
\label{tmp_cube_G0}
G(0)
=
\left (
1-\frac{2}{3}(1-\varepsilon)^2\cn^{-(2-\varepsilon)/3}
\right )^{(1-2\varepsilon)/(2-\varepsilon)}
\geq
1
-
(1+\varepsilon)
\frac{1-2\varepsilon}{2-\varepsilon}
\frac{2}{3}(1-\varepsilon)^2
\cn^{-(2-\varepsilon)/3}.
\end{equation}
On the other hand, we invoke~\eqref{eq_cube_lower_bound_1ta}, which gives
\begin{equation}
\label{tmp_cube_GR}
G(R)
\leq
1
+
\frac{1-2\varepsilon}{2-\varepsilon}
\frac{2}{3}(1-\varepsilon)^2
(R-1)
\cn^{-(2-\varepsilon)/3}.
\end{equation}
Putting~\eqref{tmp_cube_int_bound},~\eqref{tmp_cube_F_G},~\eqref{tmp_cube_G0} and~\eqref{tmp_cube_GR}
together gives~\eqref{eq_cube_bound_CW}.
\end{proof}

\begin{lemma}[Lower bound on $\CR_{\cn}$]
There exists $\cn_\star(\varepsilon) \in \N$ such that
\begin{equation}
\label{eq_cube_bound_CR}
\CR_\cn \geq (1-\varepsilon) \cn^{(1+\varepsilon)/3}
\end{equation}
for all $\cn \geq \cn_\star(\varepsilon)$.
\end{lemma}

\begin{proof}
Fix $\cn \in \N$ large enough. We recall from the its definition in~\eqref{eq_cube_definition_CR}
that $\CR_{\cn}$ is the largest integer $R$ such that $\CW_\cn^R \leq 1$. It is established
in~\eqref{eq_cube_bound_CW} that
\begin{equation*}
\CW_\cn^R \leq (R+\varepsilon) \cn^{-(1+\varepsilon)/3}.
\end{equation*}
Since the upper bound is an increasing function of $R$,
if we can find $\rho_{\cn} \in \R$ with
\begin{equation*}
(\rho_{\cn}+\varepsilon) \cn^{-(1+\varepsilon)/3} \leq 1,
\end{equation*}
then $\CR_{\cn} \geq \rho_{\cn}-1$. It is easily seen that such a value of $\rho_{\cn}$
is given by
\begin{equation*}
\rho_\cn \eq \cn^{(1+\varepsilon)/3}-\varepsilon.
\end{equation*}
We therefore have
\begin{equation*}
\CR_\cn \geq \cn^{(1+\varepsilon)/3}-\varepsilon-1 \geq \cn^{(1+\varepsilon)/3}-2 \geq (1-\varepsilon) \cn^{(1+\varepsilon)/3}
\end{equation*}
for $\cn$ large enough.
\end{proof}

\begin{lemma}[Lower bound on $\CN_{\cn}$]
There exists $\cn_\star(\varepsilon) \in \N$ such that
\begin{equation}
\label{eq_cube_bound_CN}
\CN_{\cn} \geq (1-\varepsilon)^3\cn^{(2+2\varepsilon)/3}
\end{equation}
for all $\cn \geq \cn_\star(\varepsilon)$.
\end{lemma}

\begin{proof}
Recalling the definition of $\CN_{\cn}$ in~\eqref{eq_cube_definition_CN} and
the lower bounds on $\CB_{\cn}$ and $\CR_{\cn}$ respectively established
in~\eqref{eq_cube_bound_CB} and~\eqref{eq_cube_bound_CR}, we have
\begin{align*}
\CN_{\cn}
\geq
\CB_{\CR_{\cn}+1}^\cn-n
&\geq
\left (\frac{3}{2}(1-\varepsilon)^2 \CR_{\cn} + \cn^{(2-\varepsilon)/3}\right )^{3/(2-\varepsilon)}-n
\\
&\geq
\left (\frac{3}{2}(1-\varepsilon)^3 \cn^{(1+\varepsilon)/3} + \cn^{(2-\varepsilon)/3}\right )^{3/(2-\varepsilon)}-n
\end{align*}
Factorizing out $\cn$ in the last lower bound, we have
\begin{equation*}
\CN_{\cn}
\geq
\cn \left \{
\left (1 + \frac{2}{3} (1-\varepsilon)^3\cn^{-(1-2\varepsilon)/3}\right )^{3/(2-\varepsilon)}
-
1
\right \}
\geq
\cn \left \{
\left (1 + (1-\varepsilon)^3\cn^{-(1-2\varepsilon)/3}\right ) - 1
\right \}
\end{equation*}
where we used the facts that we can restrict our attention to large values of $\cn$
and that
\begin{equation*}
(1+\tau)^{3/(2-\varepsilon)} \sim 1+3/(2-\varepsilon)\tau \geq 1+3/2 \tau
\end{equation*}
for $\tau \geq 0$ sufficiently small. Cancelling the ones and merging
the powers of $\cn$ in the last bound concludes the proof.
\end{proof}

\begin{lemma}[Lower bound on $\CS_{\cl}$]
\label{eq_cube_bound_CS}
There exists $\cl_\star(\varepsilon) \in \N$ such that
\begin{equation}
\label{eq_cube_bound_CS}
\CS_\cl
\geq
\left (
\frac{(1-2\varepsilon)^2}{3}(\cl-\cl_\star(\varepsilon)) + \CS_{\cl_\star}^{-2/(1-2\varepsilon)}
\right )^{3/(1-2\varepsilon)}.
\end{equation}
for $\cl \geq \cl_\star(\varepsilon)$. In particular, for
$\cl \geq \cl_\star(\varepsilon)/(2\varepsilon)$, we have
\begin{equation}
\label{eq_cube_bound_CS_simple}
\CS_\cl
\geq
\left ( \frac{(1-2\varepsilon)^3}{3}\cl \right )^{3/(1-2\varepsilon)}.
\end{equation}
\end{lemma}

\begin{proof}
Recalling the definition of $\CS_{\cl}$ in~\eqref{eq_cube_definition_CS} and the
lower bound on $\CN_{\cn}$ in~\eqref{eq_cube_bound_CN}, we have
\begin{equation*}
\CS_{\cl+1}
\geq
\CS_{\cl} + (1-\varepsilon)^3 \CS_{\cl}^{(2+2\varepsilon)/3}.
\end{equation*}
We therefore wish to apply Proposition~\ref{proposition_cube_bound_sequence}
with $\ck_\star = \cl_\star$ to be chosen large enough, $\rho = (1-\varepsilon)^3$,
$\alpha = (2+2\varepsilon)/3$ and $\eta = \varepsilon$. Since the sequence
$\{\CS_{\cl}\}_{\cl \geq 1}$ is unbounded, it is clear that we can select
$\cl_\star$ solely depending on $\varepsilon$ so that the
Proposition~\ref{proposition_cube_bound_sequence}, and~\eqref{eq_cube_bound_CS}
is then exactly the bound given by~\eqref{eq_cube_bound_sequence}.

To see that~\eqref{eq_cube_bound_CS_simple} holds true, we simply observe that the
condition on $\cl$ implies that $2\varepsilon \cl \geq \cl_\star(\varepsilon)$
and $(1-2\varepsilon) \cl \leq \cl-\cl_\star(\varepsilon)$.
\end{proof}

\begin{theorem}[Cube packing]
We have
\begin{equation}
\sum_{\cl \geq 1} \CH_\cl
<
+\infty.
\end{equation}
\end{theorem}

\begin{proof}
Using the definition of $\CH_{\cl}$ in~\eqref{eq_cube_definition_CH}
and the lower bound on $\CS_{\cl}$ established in~\eqref{eq_cube_bound_CS},
we have
\begin{equation*}
\CH_{\cl}
\leq
\left ( \frac{(1-2\varepsilon)^3}{3}\cl \right )^{-(1+\varepsilon)/(1-2\varepsilon)}
=
\left ( \frac{(1-2\varepsilon)^3}{3} \right )^{-(1+\varepsilon)/(1-2\varepsilon)}
\cl^{-(1+\varepsilon)/(1-2\varepsilon)},
\end{equation*}
for $\cl$ sufficiently large, and we immediately see that the series converges since
\begin{equation*}
-\frac{1+\varepsilon}{1-2\varepsilon} < -1.
\end{equation*}
\end{proof}

\bibliographystyle{amsplain}
\bibliography{bibliography}

\end{document}

%% file: general_commands.tex
\usepackage{amssymb}
\usepackage{mathrsfs}

\newcommand{\eq}{:=}

\newcommand{\supp}{\operatorname{supp}}

\newcommand{\R}{\mathbb R}

\newcommand{\Z}{\mathbb Z}
\newcommand{\N}{\mathbb N}

\newcommand{\grad}{\boldsymbol \nabla}
\renewcommand{\div}{\grad \cdot}

\newcommand{\ba}{\boldsymbol a}
\newcommand{\bb}{\boldsymbol b}

\newcommand{\bo}{\boldsymbol o}

\newcommand{\bx}{\boldsymbol x}
\newcommand{\by}{\boldsymbol y}

\newcommand{\CB}{\mathcal B}

\newcommand{\CH}{\mathcal H}

\newcommand{\CL}{\mathcal L}
\newcommand{\CM}{\mathcal M}
\newcommand{\CN}{\mathcal N}

\newcommand{\CR}{\mathcal R}
\newcommand{\CS}{\mathcal S}

\newcommand{\CW}{\mathcal W}

\newcommand{\LL}{\mathscr L}

\newcommand{\ck}{\mathcal k}
\newcommand{\cl}{\mathcal l}
\newcommand{\cm}{\mathcal m}
\newcommand{\cn}{\mathcal n}

\newcommand{\cx}{\mathcal x}
\newcommand{\cy}{\mathcal y}

%% file: special_commands.tex
\newcommand{\bzero}{\bo}

\newcommand{\tens}[1]{\underline{\boldsymbol #1}}

\newcommand{\AAA}{\tens{A}}
\newcommand{\LLL}{\tens{L}}
\newcommand{\III}{\tens{I}}
\newcommand{\JJJ}{\tens{J}}
\newcommand{\CCC}{\tens{C}}
\newcommand{\WWW}{\tens{W}}

\newcommand{\hAAA}{\widehat{\AAA}}

\newcommand{\diam}{R}

\newcommand{\hbx}{\widehat{\bx}}
\newcommand{\hby}{\widehat{\by}}

\newcommand{\hu}{\widehat u}
\newcommand{\hlambda}{\widehat \lambda}
\newcommand{\hQ}{\widehat{Q}}

\newcommand{\hmu}{\widehat{\mu}}

\newcommand{\tu}{\widetilde u}
\newcommand{\tlambda}{\widetilde \lambda}
\newcommand{\tAAA}{\widetilde{\AAA}}